\documentclass[11pt]{amsart}
\usepackage{amsmath,amssymb,amsthm}
\usepackage{amsfonts}
\usepackage{mathrsfs}
\usepackage[dvips]{graphicx}
\usepackage{color}              
\usepackage{epsfig}
\usepackage{enumerate}
\usepackage{pb-diagram,amscd,amsfonts,amsmath}

\usepackage[margin=1.3in]{geometry}
\newtheorem{theorem}{Theorem}[section]
\newtheorem{proposition}[theorem]{Proposition}
\newtheorem{corollary}[theorem]{Corollary}
\newtheorem{lemma}[theorem]{Lemma}

\theoremstyle{definition}
\newtheorem{definition}[theorem]{Definition}

\theoremstyle{remark}

\newcommand{\thmref}[1]{Theorem~\ref{#1}}
\newcommand{\propref}[1]{Proposition~\ref{#1}}

\newcommand{\eqnref}[1]{Equation~\eqref{#1}}

\newcommand{\wtl}{\widetilde}

\newcommand{\tl}{\tilde}

\newcommand{\Aut}{\operatorname{Aut}}

 \newcommand{\bphi}{{ \varphi}}

 \newcommand{\bgamma}{{\overline \gamma}}

\newcommand{\calF}{{\mathcal F}}
\newcommand{\calA}{{\mathcal A}}
\newcommand{\calB}{{\mathcal B}}

\newcommand{\calG}{{\mathcal G}}
\newcommand{\DD}{{\mathcal D}}

\newcommand{\cF}{{\mathcal F}}
\newcommand{\cG}{{\mathcal G}}

\newcommand{\NN}{{\mathcal N}}

\newcommand{\calV}{{\mathcal V}}
\newcommand{\calW}{{\mathcal W}}

\newcommand{\cA}{\mathcal{A}}
\newcommand{\CC}{{\mathbb C}}

\renewcommand{\P}{{\mathbb P}}
\newcommand{\N}{{\mathbb N}}

\newcommand{\Z}{{\mathbb Z}}
\newcommand{\ZZ}{{\mathbb Z}}

\newcommand{\from}{\colon \thinspace}
\newcommand{\ST}{{\, \, \Big| \, \,}}

\newcommand{\emul}{\stackrel{{}_\ast}{\asymp}}
\newcommand{\gmul}{\stackrel{{}_\ast}{\succ}}
\newcommand{\lmul}{\stackrel{{}_\ast}{\prec}}
\newcommand{\eadd}{\stackrel{{}_+}{\asymp}}

\newcommand{\param}{{\mathchoice{\mkern1mu\mbox{\raise2.2pt\hbox{$
\centerdot$}}
\mkern1mu}{\mkern1mu\mbox{\raise2.2pt\hbox{$\centerdot$}}\mkern1mu}{
\mkern1.5mu\centerdot\mkern1.5mu}{\mkern1.5mu\centerdot\mkern1.5mu}}}

\DeclareMathOperator{\Mod}{Mod}

\begin{document}

\title{Centralizers in Mapping Class Group and decidability of Thurston Equivalence}

 
\author{Kasra Rafi, Nikita Selinger, Michael Yampolsky}

\date{\today}

 \begin{abstract} 
We find a constructive bound for the word length of a generating set for the 
centralizer of an element of the Mapping Class Group.  
As a consequence, we show that it is algorithmically decidable
whether two postcritically finite branched coverings of the sphere are 
Thurston equivalent. 
 \end{abstract}
\maketitle

\section{Introduction}

In \cite{SY}, it has been shown that there exists an algorithm that finds a canonical 
decomposition of an obstructed Thurston map as well as geomertrization of all cycles in 
that decomposition (see Theorem~\ref{th:canonicalgeometrization} for a precise statement). 
Our motivation for the present article is to build an algorithm that can check whether 
pairwise equivalence of the pieces of canonical decompositions of two Thurston maps 
can be glued together into a global equivalence between the two.

A prototype of the algorithm has been already presented in \cite{SY} where the same result has been shown for a subclass of Thurston maps that are only allowed to have hyperbolic cycles in their canonical decompositions. This restriction significantly simplifies the problem as  the group of self-equivalences of a Thurston map with hyperbolic orbifold is trivial (which follows from the fact that an equivalence between two Thurston maps with hyperbolic orbifolds is unique \cite{DH}). We study the self-equivalence groups in the other cases. By constructively characterizing the generators of all of the groups involved, we reduce a countable search to solving a finite number of linear problems. A different approach to the problem of algorithmically verifying Thurston equivalence has been studied in \cite{BD}.

Our key result is a complexity bound on self-equivalences in the case when a first return map on a component in the canonical decomposition of a Thurston map is a homeomorphism, that is, a bound on centralizers of elements of  the Mapping Class Group. 
To accomplish this, we prove the following theorem:

\begin{theorem}
\label{centralizer theorem}
For every element $\phi$ of the Mapping Class Group, the centralizer of $\phi$ has
a generating set where every element has a word length that is bounded by
a uniform multiple of the word length of $\phi$. 
\end{theorem} 

Armed with this statement, we obtain: 

\begin{theorem}
There exists an algorithm which for any two Thurston maps $f$ and $g$ outputs an equivalence $\phi$ if $f$ and $g$ are equivalent, and outputs \textbf{maps are not equivalent} otherwise.
\end{theorem}

\bigskip 
\section{Background}

\subsection{Thurston maps}
\label{section1}

\label{section1.1}
Let $f:S^{2}\rightarrow S^{2}$ be an orientation-preserving branched covering self-map of the two-dimensional topological sphere. We define the \textit{postcritical set $P_{f}$} by
\[
 P_{f}:=\bigcup_{n>0}f^{\circ n}(\Omega_{f}),
\]
where $\Omega_{f}$ is the set of critical points of $f$. When the postcritical set $P_{f}$ is finite, we say that $f$ is \textit{postcritically finite}.

A {\it (marked) Thurston map} is a pair $(f,Q_f)$ where $f:S^2\to S^2$ is a postcritically finite ramified covering of degree at least 2 and $Q_f$ is a finite collection of marked points
$Q_f\subset S^2$ which contains $P_f$ and is $f$-invariant: $f(Q_f)\subset Q_f$. In particular, all elements of $Q_f$ are pre-periodic for $f$. 
\paragraph*{\bf Thurston equivalence.} 
Two marked Thurston maps $(f,Q_f)$ and $(g,Q_g)$ are \textit{Thurston (or combinatorially) equivalent} if there are homeomorphisms $\phi_{0}, \phi_{1}:S^{2}\rightarrow S^{2}$ such that
\begin{enumerate}
 \item the maps $\phi_{0}, \phi_{1}$ coincide on $Q_f$, send $Q_{f}$ to $Q_{g}$ and  are isotopic \text{rel } $Q_f$; 
\item the diagram
\[
\begin{CD}
S^{2} @>\phi_{1}>> S^{2}\\
@VVfV @VVgV\\
S^{2} @>\phi_{0}>> S^{2}
\end{CD}
\]
commutes.
\end{enumerate}
We will call $(\phi_0,\phi_1)$ an {\it equivalence pair}.

Let $Q$ be a finite collection of points in $S^2$.
Recall that a simple closed curve $\gamma \subset S^{2}-Q$ is \textit{essential} if it does not bound a disk, is \textit{non-peripheral} if it does not bound a punctured disk.

\begin{definition}A  \textit{multicurve} $\Gamma$ on $(S^{2},Q)$ is a set of disjoint, nonhomotopic, essential, non-peripheral simple closed curves on $S^{2}\setminus Q$.
Let $(f,Q_f)$ be a Thurston map, and set $Q=Q_f$.
A multicurve $\Gamma$ on $S\setminus Q$ is \textit{f-stable} if for every curve $\gamma \in \Gamma$, each component $\alpha$ of $f^{-1}(\gamma)$ is either trivial (meaning inessential or peripheral) or homotopic rel $Q$ to an element of $\Gamma$. 

\end{definition}

\begin{definition}
A {\it Levy cycle} is a multicurve 
$$\Gamma=\{\gamma_0,\gamma_1,\ldots,\gamma_{n-1}\}$$
such that each $\gamma_i$ has a nontrivial preimage $\gamma'_i$, where the  topological degree of $f$ restricted to $\gamma'_i$ is $1$ and $\gamma'_i$ is homotopic to $\gamma_{(i-1)\mod n}$ rel $Q$. A  Levy cycle is \textit{degenerate} if each $\gamma_i$ has a preimage $\gamma'_i$ as above such that   $\gamma'_i$ bounds a disk $D_i$  and the restriction of $f$ to $D_i$ is a homeomorphism and $f(D_i)$ is homotopic to $D_{(i+1)\mod n}$ rel $Q$.

\end{definition}

To any multicurve is associated its \textit{Thurston linear transformation} $f_{\Gamma}:\mathbb{R}^{\Gamma}\rightarrow \mathbb{R}^{\Gamma}$, best described by the following transition matrix
\[
 M_{\gamma \delta}=\sum_{\alpha} \frac{1}{\text{deg}(f:\alpha \rightarrow \delta)}
\]
where the sum is taken over all the components $\alpha$ of $f^{-1}(\delta)$ which are isotopic rel $Q$ to $\gamma$.
Since this matrix has nonnegative entries, it has a leading eigenvalue $\lambda(\Gamma)$ that is real and nonnegative (by the Perron-Frobenius theorem).

The celebrated Thurston's Theorem \cite{DH} is the following:

 \medskip\noindent
{\bf Thurston's Theorem.} {\it Let $f:S^{2} \rightarrow S^{2}$ be a marked Thurston map with    hyperbolic orbifold. Then $f$ is Thurston equivalent to rational function $g$ with a finite set of marked pre-periodic orbits if and only if $\lambda(\Gamma)<1$ for every $f$-stable multicurve $\Gamma$. The rational function $g$ is unique up to conjugation by an automorphism of $\P^1$.}

\medskip
\noindent
In view of this, an $f$-stable multicurve $\Gamma$ with $\lambda(\Gamma)\geq 1$ is called a {\it Thurston obstruction}.

In \cite{SY}, the second and third authors obtained a similar statement for Thurston maps with parabolic orbifolds:

\begin{theorem}
\label{th:degenerateLevy}
  Let $f$ be a Thurston map with postcritical set $P$ and marked set $Q\supset P$ such that the associated orbifold is parabolic and the associated matrix is hyperbolic. Then either $f$ is equivalent to a quotient of an affine map or $f$ admits a degenerate Levy cycle.

Furthermore, in the former case the affine map is defined uniquely up to  affine conjugacy.
\end{theorem}

\section{Centralizers of elements in the mapping class group}

Let $\phi$ be an element of the mapping class group $\Mod(S)$ of a surface $S$ of finite type. 
Let 
\[
C(\phi) = \Big\{ \psi \in \Mod(S) \, \Big\vert\, \psi \, \phi = \phi \, \psi \Big\} 
\] 
be the centralizer of $\phi$ in $\Mod(S)$. Fix a generating set $\calG$ for 
$\Mod(S)$ and let $\Vert \param \Vert_\calG$ denote the word 
length with respect to this generating set. For $M>0$ define 
\[
C(\phi, M) = \Big\{ \psi \in C(\phi) \ST \Vert \psi \Vert_\calG \leq M \Big\}. 
\]
In this section, we prove the following version of Theorem~\ref{centralizer theorem}:

\begin{theorem}
\label{thm:mcgc} 
There is a constant $M_0$, depending on $S$ and the generating set $\calG$, so that
for every $\phi \in \Mod(S)$, $C\big(\phi, M_0\Vert \phi \Vert_\calG \big)$ generates $C(\phi)$. 
\end{theorem}

A computational consequence of the above theorem is the following:
\begin{corollary}
  \label{cor:mcalg}
There is an algorithm which, given $\phi \in \Mod(S)$, outputs a
set of generators of $C(\phi)$. 
\end{corollary}


\subsection{Some tools}

Our main tool is the following theorem of J. Tao. 

\begin{theorem}[\cite{tao:LC}]  \label{Thm:Tao} 
For any fixed generating set $\calG$ for $\Mod(S)$, there exists a constant $K$, 
such that if $\phi, \phi' \in \Mod(S)$ are conjugate, then there is
a conjugating element $\eta$ with
\[
\Vert \eta \Vert_\calG \leq K \max (\Vert \phi \Vert_\calG + \Vert \phi' \Vert_\calG). 
\]
\end{theorem}

Let us introduce the following notations: $a \asymp_C b$, will mean $a<Cb+C$ and $b<Ca+C$, and $a\lmul b$  will mean $a<Nb$ for some fixed $N$,  and $a\emul b$  will mean $a\lmul b$ and $b\lmul a$.

We will also need the Masur-Minsky distance formula \cite{minsky:CCII}. 
For every subsurface $R \subset S$, they define a measure of complexity between 
two curve systems $d_R(\param, \param)$ called the \emph{subsurface projection
distance} (see \cite{minsky:CCII} for more details).

\begin{theorem} \label{Thm:MM} 
For any generating set $\calG$, any marking $\mu_0$, and any threshold 
$k$ that is sufficiently large, there is a uniform constant $C$ so that, for any 
$\eta \in \Mod(S)$, we have 
\[
\Vert \eta \Vert_{\calG}
\asymp_C \sum_{R \subset S}  \Big[d_R\big(\mu_0, \eta(\mu_0)\big)\Big]_k.
\]
Here the sum is over all subsurfaces $R$ of $S$, and the function $[\param ]_k$ is a truncation
function with $[x]_k = x$ when $x\geq k$ and $0$ otherwise.  
\end{theorem}

\subsection{Special cases}

\begin{proposition} \label{Prop:Finite-Order}
Theorem~\ref{thm:mcgc} holds if $\phi$ is finite order.
\end{proposition}

\begin{proof}
There are finitely many conjugacy classes of finite order elements in $\Mod(S)$
(see for example \cite[Theorem 7.13]{farb:MCG}). By \thmref{Thm:Tao}, it is sufficient to 
show that, for each
finite order element $\phi$, $C(\phi)$ is finitely generated. Indeed, consider   a 
set  $\calF$ of finite order elements by picking one representative from every conjugacy class. If each 
$C(\xi)$, $\xi \in \calF$ is finitely generated, then there is a uniform upper-bound 
$M_1$ for the word length of all elements in any such generating set. If 
$\phi$ is conjugate to $\phi' \in \calF$, there is a conjugating element $\eta$, where 
$\Vert \eta \Vert_\calG \leq K \max (\Vert \phi \Vert_\calG + \Vert \phi' \Vert_\calG)$.
Then, we can find a generating set for $C(\phi)$ by conjugating a generating set for 
$C(\phi')$. But $\Vert \phi' \Vert_\calG$ is uniformly bounded ($\calF$ is finite). 
Hence, there is $M_0$ where the word length of this generating set for $C(\phi)$
are bounded by $M_0 \Vert \phi \Vert_\calG$. 

Now let $\phi$ be any finite order element. 
To see that $C(\phi)$ is finitely generated, let $\Sigma$ be the orbifold quotient of $S$ by 
$\phi$ and $\Mod^{\rm o}(\Sigma)$ be 
the orbifold mapping class group of $\Sigma$. Then, $\Mod^{\rm o}(\Sigma)$ is a finite
index subgroup of $\Mod(\Sigma)$ and hence (say, using Schreier's lemma) 
is finitely generated. There is a finite index
sub-group of $\Mod^{\rm o}(\Sigma)$ that lifts to sub-group $C_\Sigma$ of $\Mod(S)$
(see MacLachlan and Harvey \cite[Theorem 10]{harvey:MT})  
which is also finitely generated. Finally, $C(\phi)$ is a finite extension of $C_\Sigma$
and hence is also finitely generated. 
\end{proof}

\begin{proposition} \label{Prop:pseudo-Anosov}
Theorem~\ref{thm:mcgc} holds if $\phi$ is a pseudo-Anosov element. 
\end{proposition}

\begin{proof}
By \cite{maccarthy} $C(\phi)$ is a virtually cyclic where the degree of the extension is
uniformly bounded, in particular $C(\phi)$ is finitely generated. In fact, if $F_-$ and $F_+$ 
are the stable and unstable measured
foliations associated to $\phi$, then any $\psi \in C(\phi)$ preserves the pair $(F_-, F_+)$ 
as a set. 

To prove the theorem, it is sufficient to show that, for any $\psi \in C(\phi)$ there is
a power $m$ so that $\Vert \psi \phi^m \Vert_\calG \lmul \Vert \phi \Vert_\calG$. 
Indeed, this shows that $C(\phi)$ is generated by $\phi$ and elements in $C(\phi)$
whose word length is less than a multiple of $\Vert \phi \Vert_\calG$.

We use \thmref{Thm:MM} to find such a bound. First, we claim that there exists an integer 
$m$ so that 
\begin{equation} \label{Eq:Move-Back} 
d_S(\mu_0, \psi\phi^m(\mu_0)) \lmul d_S(\mu_0, \phi(\mu_0)).
\end{equation} 
Let $\calB=\calB_\phi$ be the quasi-axis of $\phi$ in the curve graph of $S$, that is 
a geodesic in the curve graph that is preserved by a power $\phi^{m'}$ of
$\phi$ (see \cite{Bowditch08}). Then $\calB$ limits to $F_\pm$ in the boundary of the 
curve graph. And assuming $\calB$ is tight, there are only finitely many such quasi-axes 
and $\phi$ permutes them (again see \cite{Bowditch08}). 
Hence, for some power $m''$, $\psi \phi^{m''}$ also preserves $\calB$. Choose 
$m = m'' + p \cdot  m'$ so that the translation of length $\psi \phi^{m}$ along
$\calB$ is less than or equal to that of $\phi^{m'}$. Both the distance from $\mu_0$ 
to $\calB$ and the translation distance of $\phi^{m'}$ along $\calB$ are bounded by 
the word length of $\phi$. Hence, the claims follows. 

Choose such $m'$ so that the translation length of $\phi^{m'}$ is large enough to
ensure that the geodesic in the curve graph $S$ connecting 
$\mu_0$ to $\phi^{m'}(\mu_0)$ passes near $\calB$ (the curve graph is gromov hyperbolic).
Then for every subsurface $R \subsetneq S$, if $d_R\big(\mu_0, \phi^{m'}(\mu_0) \big)$ 
is large, then either 
\[
d_R\big(\mu_0, F_+ \big)
\qquad\text{or}\qquad
d_R\big(F_-, \phi^{m'}(\mu_0) \big)
\]
is large. That is, there is a constant $k$ so that 
\[
d_R\big(\mu_0, \phi^{m'}(\mu_0) \big)  \geq \min \Big( d_R\big(\mu_0, F_+ \big), 
d_R\big(F_-, \phi^{m'}(\mu_0) \big) \Big)  - k. 
\]
Using $d_R\big(F_-, \phi^{m'}(\mu_0) \big) = d_R\big(F_-, \mu_0\big)$ we get
\[
\Big[ d_R\big(\mu_0, \phi^{m'}(\mu_0) \big) \Big]_{2k}  \gmul 
\Big[d_R\big(\mu_0, F_+ \big) \Big]_k + 
\Big[ d_R\big(F_-, \mu_0 \big) \Big]_k. 
\]
Hence,
\begin{align} \label{Eq:F}
\Vert \phi \Vert_\calG \emul \Vert \phi^{m'} \Vert_\calG 
& \asymp_C \sum_{R \subset S}  \Big[d_R\big(\mu_0, \eta(\mu_0)\big)\Big]_{2k} \notag \\
& \gmul  d_S(\mu_0, \phi^{m'}(\mu_0)) + 
\sum_{R \subsetneq S} \Big[ d_R(\mu_0, F_-) \Big]_k+ 
\sum_{R \subsetneq S} \Big[ d_R(\mu_0, F_+) \Big]_k. 
\end{align}

Now, let $m$ be as in \eqref{Eq:Move-Back} and let $\xi = \psi \phi^m$. 
Further assume $k$ is large enough so that $d_R(F_-, F_+) < k$. Then
\begin{align*} 
\Vert \xi \Vert_\calG &\emul d_S(\mu_0, \xi(\mu_0)) + 
\sum_{R \subsetneq S} \Big[ d_R(\mu_0, \xi(\mu_0) \Big]_{2k}\\
& \lmul d_S(\mu_0, \xi(\mu_0)) + 
\sum_{R \subsetneq S} \Big[ d_R(\mu_0, F_+) \Big]_k
+ \Big[ d_R(\xi(\mu_0), F_+ )\Big]_k
\end{align*}
If $\xi(F_+) = F_+$. 
\[
\sum_{R \subsetneq S} \Big[ d_R(\xi(\mu_0), F_+ )\Big]_k = 
\sum_{R \subsetneq S} \Big[ d_R(\mu_0, F_+ )\Big]_k
\]
and if $\xi(F_+) = F_-$. 
\[
\sum_{R \subsetneq S} \Big[ d_R(\xi(\mu_0), F_+ )\Big]_k = 
\sum_{R \subsetneq S} \Big[ d_R(\mu_0, F_- )\Big]_k
\]
In either case, the last two terms in estimate above given for $\Vert \xi \Vert_\calG$ are
less than the lower-bound \eqnref{Eq:F} given for $\Vert \phi \Vert_\calG$.
We also know from \eqnref{Eq:Move-Back} that the first term is bounded above by $d_S(\mu_0, \phi(\mu_0))$ which is also bounded by above by a multiple of 
$\Vert \phi \Vert_\calG$. The Theorem follows. 
\end{proof}

\subsection{The general case} 

Recall from the Nielsen-Thurston classification of surface homeomorphisms \cite{thurston:GD, flp:TTS} 
that there is a normal form for any homeomorphism $\phi$ of a surface $S$ of finite type. That is, 
\begin{enumerate}
\item There is a multicurve $\Gamma_\phi$ that is preserved by $\phi$, called the 
\emph{canonical reducing system}, defined as follows: consider the set $\calA_\phi$
consisting of all curves $\alpha$ so that $\phi^k(\alpha) = \alpha$ up to isotopy, for some $k>0$ and let 
$\Gamma_\phi$ be the boundary of the subsurface of $S$ that is filled with curves in 
$\calA_\phi$. The curve system $\Gamma_\phi$ is empty if $\phi$ is pseudo-Anosov or 
has finite order. 
\item The components of $S - \Gamma_\phi$ are decomposed into \emph{$\phi$--orbits} 
$\{ V_1, \dots, V_\ell\}$ where 
\[
\phi(V_i) = V_{i+1} \qquad\text{for}\qquad i \in \Z / \ell \, \Z.
\] 
\item For every every $\phi$--orbit $\calV=\{ V_1, \dots, V_\ell\}$ the first return map
$\phi^\ell \from V_1 \to V_1$ is either finite order or pseudo-Anosov. 
\end{enumerate}

It is convenient to fix a topological surface $\delta$ that is homeomorphic to every $V_i$. 
Choosing a homeomorphism $\delta \to V_i$, the map 
\[
\phi^\ell |_{V_i} \from V_i \to V_i
\]
defines a conjugacy class in $\Mod(\delta)$ that is independent on $i$ or the homeomorphism 
from $\delta$ to $V_i$. That is, it depends only on the $\phi$--orbit $\calV$. We
denote this conjugacy class by $[\phi_\calV]$. We say the $\phi$--orbit $\calV$ is 
\emph{of type $\delta$ with the first return map $[\phi_\calV]$}. 

%

%

We start by modifying the generating set $\calG$ and conjugating $\phi$ so
that they are compatible with each other. 

If we choose $\bphi \in [\phi]$ with $\Vert \bphi \Vert_\calG \lmul  \Vert \phi \Vert_\calG$ 
then, by \thmref{Thm:Tao}, the conjugating element $\eta$ satisfies 
$\Vert \eta \Vert_\calG \lmul \Vert \phi \Vert_\calG$. 
But $\eta$ conjugates a generating set for $C(\bphi)$ to a generating
set for $C(\phi)$ which means it would be enough to prove the theorem for $\bphi$. 
Our goal is to find a representative of the conjugacy class $[\phi]$ of $\phi$ which has 
(as much as possible) a standard form.
\medskip

There are finitely many topological types possible for subsurfaces of $S$. 
Let $\Delta$ be the set of surfaces that can be a subsurface of $S$. That is, 
for every subsurface $V$ of $S$, there is a (unique) surface $\delta \in \Delta$
that is homeomorphic to $V$. We fix a generating set $\calG_\delta$ for every surface 
$\delta \in \Delta$. 
In fact, we assume $\calG_\delta$ consists of Dehn twists around a finite set of 
curves $\mu_\delta$. Curves in $\mu_\delta$ fill the surface $\delta$, that is 
every curve in $\delta$ intersects a curve in $\mu_\delta$.

Also, up to a homeomorphism, there are finitely many multicurves on a surface $S$. 
Let $\Lambda$ be a fixed set consisting of a representative for every 
homeomorphism type of a multicurve in $S$. For any simple closed curve
$\gamma$ in $S$, let $D_\gamma$ denote the Dehn twist around $\gamma$.  
For each $\Gamma \in \Lambda$ let $\mu_\Gamma$ 
be a set of curves on $S$ with the following properties. 
\begin{enumerate}
\item $\Gamma \subset \mu_\Gamma$. 
\item The set $\calG_\Gamma =\{ D_\gamma \, \vert\, \gamma \in \mu_\Gamma\}$ 
generates $\Mod(S)$. 
\item for every subsurface $V$ that is a component of $S - \Gamma$ that is homeomorphic
to $\delta$, there is a homeomorphism $m_V \from \delta \to V$ so that 
$m_V(\mu_\delta)$ is exactly the set of curves in $\mu_\Gamma$ that are contained in $V$. 
In particular, $m_V(\calG_\delta) \subset \calG_\Gamma$ generates $\Mod(V)$. 
\end{enumerate}
Note that, $\mu_\Gamma$ fills $S$. For the rest of this article, we assume 
\[
\calG_0 = \bigcup_{\Gamma \in \Lambda} \calG_\Gamma
\qquad\text{and}\qquad 
\mu_0 = \bigcup_{\Gamma \in \Lambda} \mu_\Gamma. 
\]
Note that $\Vert \param \Vert_\calG$ differs from $\Vert \param \Vert_{\calG_0}$ 
by a uniform multiplicative amount. 

We are now ready to construct $\bphi$. Let $\Gamma \in \Lambda$ be the curve
system in $\Lambda$ that has the same homeomorphism type as $\Gamma_\phi$.
Conjugate $\phi$ to $\phi'$ by a homeomorphism that sends $\Gamma_\phi$ to 
$\Gamma$. Then, $\phi'$ partitions the components of $S - \Gamma$ to $\phi'$--orbits 
similar to $\phi$. We then further modify $\phi'$ to $\bphi$ whose orbits are the same
as the orbits of $\phi'$ so that, if $\calV'=\{ V_1', \dots, V_\ell'\}$ is a $\phi'$--orbit of size $\ell$ 
associated to the $\phi$--orbit $\calV$, then
\begin{enumerate}
\item for $i = 1, \ldots, \ell-1$, we have 
\[
m_{V_{i+1}}^{-1}\, \bphi \,  m^{\ }_{V_1} \from \delta \to \delta 
\qquad \text{is the identity map} 
\]
\item The map 
\[
m_{V_1}^{-1} \, \bphi \,  m^{\ }_{V_\ell} \from \delta \to \delta 
\]
is the representative of $[\phi_\calV]$ that has the shortest word length with respect to 
$\calG_\delta$. We can make this canonical by choosing, ahead of time, a representative 
for every conjugacy class in $\Mod(\delta)$. 
\end{enumerate} 

\begin{proposition} 
For $\bphi$ constructed as above, we have
\[
\Vert \bphi \Vert_{\calG_0} \lmul  \Vert \phi \Vert_\calG
\]
\end{proposition} 

\begin{proof} 
The proposition follows from the Masur-Minsky distance formula (\thmref{Thm:MM}), which we
apply to $\bphi$. Since $\bphi (\Gamma) = \Gamma$, for every $R$ that
intersects $\Gamma$, we have 
\[
d_R(\mu, \bphi(\mu)) \eadd d_R(\Gamma, \bphi(\Gamma)) = O(1). 
\]
Hence, choosing $k$ large enough, these terms disappear from the distance formula. 

Also, for every $\bphi$--orbit, $\calV'$ associated to the $\phi$--orbit $\calV$, 
we have,
\begin{align*} 
\big\Vert \bphi^\ell\big|_\calV \big\Vert_{\calG_\delta} 
 &\emul \sum_{\rho \subset \delta} \Big[ d_\rho(\mu_\delta, \phi_\calV(\mu_\delta))\Big]_k\\
 &\emul \min_{\mu \in M_\delta} 
     \sum_{\rho \subset \delta} \Big[ d_\rho(\mu, \phi_\calV(\mu))\Big]_k\\
 & = \min_{\mu \in M_{V_1}} 
     \sum_{R \subset V_1} \Big[ d_R(\mu, \phi^\ell(\mu))\Big]_k\\ 
 &\lmul \sum_{R \subset V_1} \Big[ d_R(\mu_0, \phi^\ell(\mu_0))\Big]_k\\
 &\lmul \sum_{R \subset S} \Big[ d_R(\mu_0, \phi^\ell(\mu_0))\Big]_k 
   \emul \Vert \phi^\ell \Vert_{\calG_0} \leq \ell  \Vert \phi \Vert_{\calG_0}
\end{align*}

%
%
%
%

Now, let $\mu_S$ be a marking for $S$ associated to the generating set $\calG_S$
and let $\mu_0$ be the marking in $\delta$ that is the image of the projection of 
$\mu_S$ to $V_1$ under $m_{V_1}^{-1}$. That is, for every sub-surface $R$ of $\delta$,
we have
\[
d_{m_{V_1}\!(R)}\big(\mu_S, m_{V_1}(\mu_0)\big) = O(1). 
\]
We can now compare the word length of $\phi_\calV$ with that of $\phi^\ell$
which send $V_1$ to itself. 
\begin{align*}
\Vert \phi_\calV \Vert_{\calG_\delta} 
  & \prec \sum_{R \subset \delta}  \Big[d_R\big(\mu_0, \phi_\calV(\mu_0)\big)\Big]_k\\
  & \emul \sum_{R \subset V_1}  \Big[d_R\big(\mu_S, \phi_\calV(\mu_S)\big)\Big]_k\\
  & \prec \sum_{R \subset S}  \Big[d_R\big(\mu_S, \phi^\ell(\mu_S)\big)\Big]_k
   \asymp \Vert \phi^\ell \Vert_{\calG_S}
\end{align*} 
where the last inequality is the distance formula in the surface $S$. The lemma
now follows since $\Vert \phi^\ell \Vert_{\calG_S} \leq \ell \, \Vert \phi \Vert_{\calG_S}$ and
$\ell$ and all other related constant are independent of $\phi$. 
\end{proof}

\begin{proposition} \label{Prop:subsurface} 
There is a constant $M_\delta$, depending only on the generating sets 
$\calG_\delta$ and $\calG$, so that for every $\phi \in \Mod(S)$ that has a $\phi$ orbit 
$\calV =\{ V_1, \dots, V_\ell\}$ of type $\delta$, with the first return map 
$\phi_\calV$, we have 
\[
\Vert \phi_\calV \Vert_{\calG_\delta} \leq M_\delta \, \Vert \phi \Vert_\calG. 
\]
\end{proposition}

\begin{proof} 
If $\phi_\calV$ is finite order, then the statement is clear since there are
only finitely many conjugacy classes of finite order elements and $\phi_\calV$ is
one of the finitely many fixed representatives of these classes. 
Hence, $\Vert \phi_\calV \Vert_{\calG_\delta}$ is uniformly bounded. However, we give 
a general argument that works in both cases using the Masur-Minsky 
distance formula Therorem~\ref{Thm:MM}. 


Note that the changing $\eta$ by a conjugation is the same as changing the marking
$\mu_\delta$. Since we have chosen $\phi_\calV$ to be the representative
$[\phi_\calV]$ with the smallest word length, we have,
\[
\Vert \phi_\calV \Vert_{\calG_\delta}
\prec \sum_{R \subset \delta}  \Big[d_R\big(\mu, \phi_\calV(\mu)\big)\Big]_k
\qquad \text{ for any marking $\mu$}. 
\]
Here, $\prec$ means less than up a uniform multiplicative and additive error. 

Now, let $\mu_S$ be a marking for $S$ associated to the generating set $\calG_S$
and let $\mu_0$ be the marking in $\delta$ that is the image of the projection of 
$\mu_S$ to $V_1$ under $m_{V_1}^{-1}$. That is, for every sub-surface $R$ of $\delta$,
we have
\[
d_{m_{V_1}\!(R)}\big(\mu_S, m_{V_1}(\mu_0)\big) = O(1). 
\]
We can now compare the word length of $\phi_\calV$ with that of $\phi^\ell$
which send $V_1$ to itself. 
\begin{align*}
\Vert \phi_\calV \Vert_{\calG_\delta} 
  & \prec \sum_{R \subset \delta}  \Big[d_R\big(\mu_0, \phi_\calV(\mu_0)\big)\Big]_k\\
  & \emul \sum_{R \subset V_1}  \Big[d_R\big(\mu_S, \phi_\calV(\mu_S)\big)\Big]_k\\
  & \prec \sum_{R \subset S}  \Big[d_R\big(\mu_S, \phi^\ell(\mu_S)\big)\Big]_k
   \asymp \Vert \phi^\ell \Vert_{\calG_S}
\end{align*} 
where the last inequality is the distance formula in the surface $S$. The lemma
now follows since $\Vert \phi^\ell \Vert_{\calG_S} \leq \ell \, \Vert \phi \Vert_{\calG_S}$ and
$\ell$ and all other related constant are independent of $\phi$. 
\end{proof}

Now, consider an element $\psi \in C(\phi)$. First notice that if $\alpha \in \calA_\phi$ then 
\[
\phi^k (\psi(\alpha)) = \psi(\phi^k(\alpha)) = \psi(\alpha),
\]
which means $\psi(\calA_\phi)=\calA_\phi$. Hence, $\psi$ also preserves the subsurface 
that is filled with the curves in $\calA_\phi$. Therefore, 
$\psi(\Gamma_\phi) = \Gamma_\phi$ and $\psi$ permutes the components of 
$S - \Gamma_\phi$. 

Assume, $\psi(V_1) = W_1$ where $V_1$ is in a $\phi$--orbit $\calV$ and $W_1$ is 
in a $\phi$--orbit $\calW$. We observe that $\calV$ and $\calW$ are of the same type 
$\delta$. Also, for every $i$, 
\[
\psi(V_i) = \psi( \phi^i (V_1) ) = \phi^i(\psi(V_1))= \phi^i(W_1)= W_i.
\]
Hence, the orbit $\calV$ is mapped to the $\calW$, which in particular implies 
$|\calV| = |\calW|=\ell$. We further have
\[
\psi|_{V_1} \, \phi|_{V_1}^\ell \, \psi|_{V_1}^{-1}= \phi|_{W_1}^\ell, 
\]
where $\psi|_{V_1}$ is the restriction of $\psi$ to $V_1$ (similarly, $\phi|_{V_1}$ 
and $\phi|_{W_1}$ are restrictions of $\phi$ to $V_1$ and $W_1$ respectively). 
That is, $[\phi_\calV]= [\phi_\calW]$ which implies $\phi_\calV=\phi_\calW$.

To sum up, $\psi$ induces a permutation of $\phi$--orbits, however, it can only send an 
orbit to another orbit if the orbits have the same size, same topological type and if the
associated first return maps are the same.  Also, $\psi$ has to send adjacent 
components of $S - \Gamma_\phi$ to adjacent components. To keep track of this 
information, we
consider the a decorated dual graph defined as follows. Let $G$ be a graph whose vertices 
are components of $S - \Gamma_\phi$ and edges are pairs of adjacent components. 
We decorate a vertex $V \in G$ (which is a component of $S - \Gamma_\phi$) 
with the name $\calV$ of the associated $\phi$--orbit, 
the topological type $\delta$ and the first return map $\phi_\calV \in \Mod(\delta)$. 
We say a map $f \from G \to G$ is an automorphism of the decorated graph if
\begin{enumerate}
\item $f$ is a graph automorphism. 
\item  There is a permutation $\sigma$ of the $\phi$--orbits so that, if 
\[
f(V, \calV, \delta, \phi_\calV)= (W, \calW, \delta', \phi_\calW)
\]
then $\delta = \delta'$, $|\calV|= |\calW|$, $\calW = \sigma(\calV)$ and 
$\phi_\calV = \phi_\calW$. 
\end{enumerate}
Note that, the set of automorphisms of the decorated graph $G$ form a group which we denote
by $\Aut(G)$. For a given
$\psi \in C(\phi)$, we denoted the induced graph map by $f_\psi$ and the induced permutation
of $\phi$--orbits by $\sigma_\psi$. We have a homomorphism 
\[
\pi_G \from C(\phi) \to \Aut(G)
\]
projecting $\psi$ to the induced action $f_\psi$ on the decorated graph $G$.

For each orbit $\calV$, let 
\[
C_\calV \subset \ker(\pi_G)
\]
be the set of elements of $C(\phi)$ that fix every subsurface in $S -\Gamma_\phi$ 
whose restriction to any subsurface that is not in $\calV$ is identity. Then 
the $\cup_\calV C_\calV$ generates $\ker(\pi_G)$ and the intersection 
$\cap_\calV C_\calV$ is the set of multi-twists around the curves $\Gamma$. 
Also, for $\psi \in C_\calV$ where $\calV=\{V_1, \dots, V_\ell\}$, 
\[
\psi|_{V_{i+1}} = \phi^{-i} \psi|_{V_1} \phi^{i}. 
\]
That is, the restriction $\psi|_{V_1}$ determines the restriction of $\phi$
to every other subsurface in the $\phi$--orbit $\calV$. Therefore, considering
the homeomorphism 
\[
\pi_\calV \from C_\calV \to \Mod(\delta), 
\qquad \psi \to   m_{V_1}^{-1} \psi|_{V_1}\, m_{V_1}
\]
we have that any $\psi \in C_\calV$ is determined, up to possibly a multi-twist around 
$\Gamma$, by its projection $\pi_\calV(\phi)$ to $\Mod(\delta)$ which lies in 
$C(\phi_\calV)$. 

We have shown that every element of $C(\phi)$ is determined, up a multi-twist around $
\Gamma$, by its projection to $\Aut(G)$ and to $C(\phi_\calV)$. 
We now examine which multi-twists around curves in $\Gamma$ lie in $C(\phi)$. 
Consider the action of 
$\phi$ on $\Gamma$. It decomposes $\Gamma$ into orbits
$\bgamma = \{\gamma_1, \dots, \gamma_j \}$ where 
$\phi(\gamma_i) = \gamma_{i+1}$ for $i \in \Z/j\Z$. We call such orbit
$\bgamma$ an admissible multicurve if $\phi^{\, j}$ sends $\gamma_1$ to
$\gamma_1$ preserving the orientation. For any admissible multicurve 
$\bgamma$, define
\[
D_\bgamma = D_{\gamma_1} \dots D_{\gamma_j},
\]
where $D_{\gamma_i}$ is a Dehn twist if $\gamma_i$ is non-separating
and a half-twist if $\gamma_i$ is separating. That is, $D_\bgamma$ is the product 
of Dehn twists (or half-twists) around the curves in $\bgamma$. 
The set of multi-twists around the curve in $\Gamma$ that commute with 
$\phi$ is generated by $\{D_\bgamma \}_\bgamma$. Note that this maybe
an empty set. 
This is because, if an element $\psi \in C(\phi)$ twists around $\gamma$
is also has to twist by the same amount around $\phi(\gamma)$. However, 
if $\phi^{\, j}$ sends $\gamma$ to itself reserving the orientation, then 
$\phi^{\, j}$ conjugates $D_\gamma$ to $D_\gamma^{-1}$. Hence, 
$D_\gamma$ does not commute with $\phi$ and no Dehn twists around
such $\gamma$ is possible. 

There is no homomorphism back from $\Aut(G)$ or $C(\phi_\calV)$
to $\Mod(S)$. But to find a generating set for $C(\phi)$ it is enough to choose a section. 
To summarize the above discussion, we have shown:

\subsection*{Summary} Let $\calG_\calV$ be a generating set for $C(\phi_\calV)$ and
consider arbitrary sections 
\[
\pi_G^{-1} \from \Aut(G) \rightsquigarrow C(\phi)
\qquad\text{and}\qquad
\pi_\calV^{-1} \from C(\phi_\calV) \rightsquigarrow C_\calV. 
\]
Then $C(\phi)$ is generated by the union of the following sets:
\begin{enumerate}
\item The set $\{D_\bgamma, \gamma \in \Gamma\}$,  where
$\bgamma \subset \Gamma$ is an admissible multicurve. 
\item The image of $\Aut(G)$ under $\pi_G^{-1}$ 
\item The images of $\calG_\calV$ under maps $\pi_\calV^{-1}$.
\end{enumerate}

What remains is to bound the word length of the elements of this
generating set. An element in (1) is a product of Dehn twists around 
a uniformly bounded number of curves and these Dehn twists are already in 
our generating set. For (2), we build the section to be as close to the identity as possible. 
Namely, for any $f \in \Aut(G)$ and induced permutation $\sigma$,  let
$\sigma(V_i) = W_j$ where $V_i$ in the $\phi$-orbit $\calV$ and $W_j$ in 
a $\phi$--orbit $\calW$. We define $\psi$ to be the map that also sends $V_i$ to 
$W_j$ and so that $m_{W_j}^{-1}\psi m_{V_i}$ is the identity. Then $\psi$
is clearly in $C(\phi)$. For (3), given an element $g \in \calG_\calV$ associated to 
a $\phi$-orbit $\calV= \{V_1, \dots, V_\ell\}$ there is mapping class $\psi$, that acts on
subsurfaces of $S - \calA_\phi$ the same way as $\phi$, its restriction to $\calV$
is the same as $\phi$ and is the identity on every other orbit. Again, $\psi$ clearly
commutes with $\phi$. The desired upper-bound for the word length of $\psi$ follows 
from \propref{Prop:subsurface}.

\section{Self-equivalences of Thurston maps}

The results of the previous section imply the following theorem.

\begin{theorem}
\label{thm-centralizer}
Let $f$ be either a Thurston map with empty canonical obstruction or a homeomorphism. Then the group $C(f)$ of all self-equivalences of $f$ is finitely generated. Moreover, there is an algorithm that finds a generating set for $C(f)$.
\end{theorem}

\begin{proof}

I. If $f$ is an unobstructed Thurston map with   hyperbolic orbifold then it is equivalent to a rational map (possibly with extra marking) and $C(f)$ is trivial (cf. \cite{DH,BGL}).  The same argument applies in the case of a Thurston map with  parabolic orbifold unless $P_f$ contains exactly $4$ points, and $f$ is equivalent to a quotient of an affine map $Ax+b$ with hyperbolic associated matrix $A$ (see \cite{SY}). 

II. Let $f$ be a Thurston map with   parabolic orbifold  such that $C(f)$ is non-trivial.   If $Q_f=P_f$ contains exactly 4 points, then the pure mapping class group of $(S^2, Q_f)$ is isomorphic to the modular group $\Lambda=\text{PGL}(2,\Z) / \text{PGL}(2,\Z/2\Z)$. In this case, $C(f)$ is the subgroup of $\Lambda$ of all matrices that commute with $A$. It consists of the matrices which diagonalize simultaneously with $A$, and thus its generating set can be easily computed. 

If $Q_f$ has more than 4 points, let us denote $C(f,P_f)$ the group of self-equivalences of $f$ with only the points in $P_f$ marked. Clearly,  $C(f)$ is isomorphic to a finite index subgroup of $C(f,P_f)$. Indeed, if a self-equivalence $\phi$ is homotopic to the identity in $(S^2, Q_f)$ it will also be homotopic to the identity in $(S^2, P_f)$. Therefore every self-equivalence can be represented by an affine homeomorphism. Some elements of $C(f,P_f)$, however, may have affine representatives that do not fix points in $Q_f$ but instead send them to different pre-periodic orbits. Determining which subgroup of $C(f,P_f)$ fixes points in $Q_f$ is a straightforward exercise in linear algebra.

III. If $f$ is a homeomorphism then Corollary~\ref{cor:mcalg} can be applied.
\end{proof}

\section{Hurwitz classification of branched covers}

Let $X$ and $Y$ be two finite type Riemann surfaces. We recall that two finite degree branched covers $\phi$ and $\psi$ of
$Y$ by $X$ are {\it equivalent in the sense of Hurwitz} if there exist  homeomorphisms $h_0:Y\to Y,h_1:X\to X$ such that
$$h_0\circ \phi=\psi\circ h_1.$$
An equivalence class of branched covers is known as a {\it Hurwitz class}. Enumerating all Hurwitz classes with a given
ramification data is a version of the {\it Hurwitz Problem}. The classical paper of Hurwitz \cite{hur}
gives an elegant and explicit solution of the problem for the case $X=\hat\CC$.

We will need the following narrow consequence  of Hurwitz's work  (for a modern treatment, see \cite{barth}):
\begin{theorem}
\label{th:hurwitz}
There exists an algorithm $\cA$ which, given PL branched covers $\phi$ and $\psi$ of PL spheres and a PL 
homeomorphism $h_0$ mapping the
critical values of $\phi$ to those of $\psi$, does the following:
\begin{enumerate}
\item decides whether $\phi$ and $\psi$ belong to the same Hurwitz class or not;
\item if the answer to (1) is affirmative, decides whether there exists a homeomorphism $h_1$ such that $h_0\circ \phi=\psi\circ h_1.$
\end{enumerate}
\end{theorem}

\section{Equivalence on thick parts}

\subsection{Canonical obstructions and thin-thick decompositions of Thurston maps.}

Let $f$ be a Thurston map, and $\Gamma=\{\gamma_j\}$ an $f$-stable multicurve. Consider   a finite collection of 
disjoint closed annuli
$A_{0,j}$ which are homotopic to the respective $\gamma_j$. For each $A_{0,j}$ consider only non-trivial preimages; these form
a collection of annuli $A_{1,k}$, each of which is homotopic to one of the curves in $\Gamma$.
Following Pilgrim, we say that the pair $(f,\Gamma)$
 is in a {\it standard form} (see Figure \ref{fig-decomp})
if there exists a collection of annuli $A_{0,j}$, which we call {\it decomposition annuli}, as above such that the
following properties hold:
\begin{itemize}
\item[(a)] for each curve $\gamma_j$ the annuli $A_{1,k}$ in the same homotopy class are contained inside $A_{0,j}$;
\item[(b)] moreover, the two outermost annuli $A_{1,k}$ as above share their outer boundary curves with $A_{0,j}$.
\end{itemize}

\begin{figure}[ht]
\includegraphics[width=\textwidth]{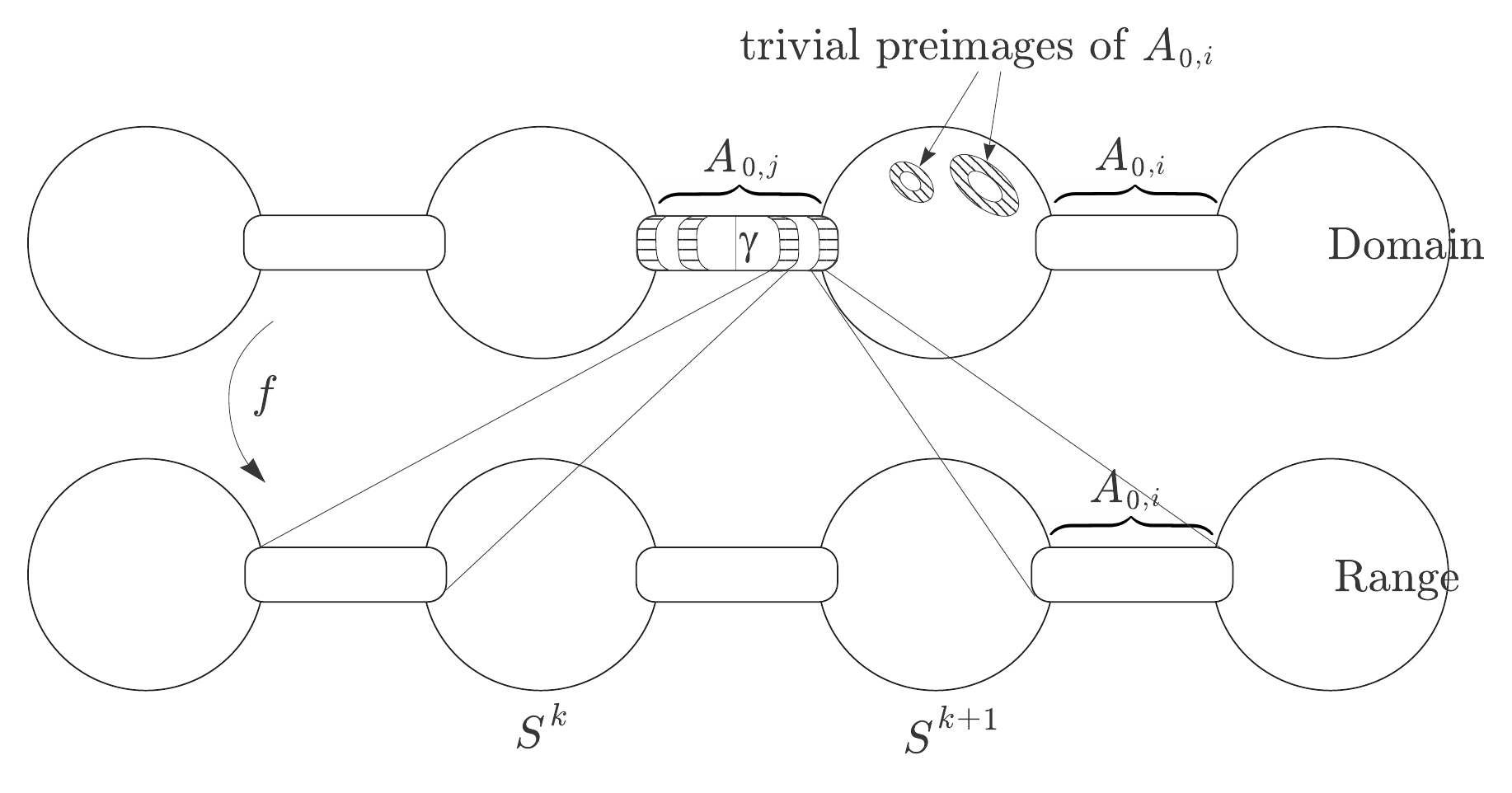}
\caption{\label{fig-decomp}Pilgrim's decomposition of a Thurston map}
\end{figure}

A Thurston map with a multicurve 
in a standard form can be decomposed as follows. First, all annuli $A_{0,j}$ are removed, leaving a collection of spheres with holes, denoted $S_0(j)$. For each $j$, there exists a unique connected component $S_1(j)$ of
$f^{-1}(\cup S_0(j))$ which has the property $\partial S_0(j)\subset \partial S_1(j)$. Any such component $S_1(j)$ is a sphere with
holes, with boundary curves being of two types: boundaries of the removed annuli, or boundaries of trivial preimages of the
removed annuli. 

The holes in $S_0(j)\subset S^2$ can be filled as follows. Let $\chi$ be a boundary curve of
a component $D$ of $S^2\setminus S_0(j)$. 
Let $k\in\NN$ be the first iterate $f^k:\chi\to\chi$, if it exists. For each $0\leq i\leq k-1$ the curve 
$\chi_i\equiv f^i(\chi)$ bounds a component $D_i$ of $S^2\setminus S_0(m_i)$ for some $m_i$. Denote $d_i$ the degree
of $f:\chi_i\to \chi_{i+1}$.  Select  homeomorphisms 
$$h_i:\bar D_i\to \bar \DD\text{ so that }h_{i+1}\circ f\circ h_i^{-1}(z)=z^{d_i}.$$
Set $\tilde f\equiv f$ on $\cup S_0(j)$.
Define new punctured spheres $\tilde S(j)$ by adjoining cups $h_i^{-1}(\bar \DD\setminus \{0\})$
to $S_0(j)$. Extend the map $\tilde f$ to each $D_i$ by setting 
$$\tilde f(z)=h_{i+1}^{-1}\circ (h_i(z))^{d_i}.$$
We have thus replaced every hole with a cap with a single puncture. We call such a procedure {\it patching} a component.

By construction, the map 
$$\tilde f:\cup \tilde S(j)\to\cup\tilde S(j)$$
contains a finite number of periodic cycles of  punctured spheres. For every periodic sphere $\tilde S(j)$ denote
by $\cF$ the first return map $\tl f^{k_j}:\tilde S(j)\to\tilde S(j).$ This is again a Thurston map or a homemorphism.
The collection of maps $\cF$ and the combinatorial information required to glue the spheres $S_0(j)$ back together 
is what Pilgrim called a {\it decomposition} of $f$ along $\Gamma$; we will denote it $\mathcal S_\Gamma$.

Pilgrim showed:
\begin{theorem}
\label{th:decomp1}
For every obstructed marked Thurston map $f$ with an obstruction $\Gamma$ there exists an equivalent map $g$ 
such that $(g,\Gamma)$ is in a standard form, and thus can be decomposed.
\end{theorem} 

Pilgrim \cite{Pil1} defined a {\it canonical decomposition} of a Thurston map $f$ based on his definition of a canonical Thurston obstruction. His original definition was framed in the language of iteration on a Teichm{\"u}ller space; we will give an equivalent definition discovered by the second author  \cite{seljmd}:
\begin{theorem}
Suppose $f$ is an obstructed Thurston mapping. Then
there exists a unique minimal (with respect to inclusion) obstruction $\Gamma_f$, which is the canonical obstruction in the sense of \cite{Pil1}, with the following properties. 
  	\begin{itemize}
		\item If a first-return map $\cF$ of a cycle of components  in $\mathcal S_{\Gamma_f}$ is a $(2,2,2,2)$-map, then  every curve of every simple Thurston obstruction for $\cF$ has two postcritical points of $f$ in each complementary component and the two eigenvalues of $\hat{\cF}_*$ are equal or non-integer.
		\item If the first-return map $\cF$ of   a cycle of components  in $\mathcal S_{\Gamma_f}$ is not a $(2,2,2,2)$-map nor a homeomorphism, then there exists no Thurston obstruction for $\cF$.
	\end{itemize}
\end{theorem}

\begin{definition}
If $\Gamma_f$ is the canonical obstruction, then the decomposition $\mathcal S_{\Gamma_f}$ is the {\it canonical decomposition} of $f$. In this case,
we call the components of the complement of the decomposition annuli the {\it thick parts}, and the decomposition annuli themselves the {\it thin parts}.
\end{definition}
From this point, only canonical decompositions of Thurston maps will be considered.

\begin{definition}
By \emph{equivalence on thick parts} $\phi$ between $f$ and $g$ we mean a homeomorphism defined on the union of patched thick parts of $f$ onto the union of patched thick parts of $g$ such that the following holds:

\begin{itemize}
\item Denote $\phi_W$ the restriction of $\phi$ to any patched thick component $W$.  
 If $X$ is a periodic patched thick component of $\tl f$ then $Y=\phi(X)$ is periodic for $\tl g$ with the same period. If  $\cF_X$, $\cG_Y$ denote the first return maps of $X$ and $Y$ respectively, then $\phi_X$ is an equivalence of $\cF_X$ and $\cG_Y$. 
\item Let $X$ be a periodic patched thick component and let  $W$ be a preimage of $X$ so that
$$\tl f^n(W)=X. $$
Denote $Y=\phi_X(X)$ and $Z=\phi(W)$. Then $\tl g^n(Z)=Y$ and $\phi_W$ is a lift of $\phi_X$  through actions of $\tl f$ and $\tl g$:
$$\phi_X\circ \tl f^n =\tl g^n\circ\phi_W.  $$
\end{itemize}
\end{definition}

\subsection{Centralizer on thick parts.}
For each periodic patched thick component $X=\cF(X)$ of the canonical decomposition of $f$, denote $C_X(f)\subset \text{PMCG}(X)$ the 
group of self-equivalences of the first return Thurston mapping $\cF|_X:X\to X$.

\begin{definition}
We define the \emph{centralizer on thick parts} $C_\text{thick}(f)$ of $f$ to be the group of all self-equevalences of $f$ on thick parts. By the previous definition, $C_\text{thick}(f)$ is isomorphic to the subgroup of the free abelian product $$C_\text{periodic}(f)\equiv \prod_{\text{periodic components }X} C_X(f)$$ consisting of all elements $\phi$ such that 
for every thick patched component $W$ with $\tl f^n(W)=X$, one can define $\phi_W$ so that $\phi_X\circ \tl f^n =\tl f^n\circ\phi_W $ (that is, 
$\phi$ can be lifted via the action of $\tl f$ to all strictily pre-periodic preimages of $X$).


\end{definition}

Note that since all $C_X$ are finitely generated (Theorem~\ref{thm-centralizer}) and $C_\text{thick}(f)$ is a subgroup of finite index, $C_\text{thick}(f)$ is also finitely generated.
Furthermore,
\begin{lemma}
\label{lem:selfeq1}
A generating set of $C_\text{thick}(f)$ can be computed explicitly.
\end{lemma}
\begin{proof}
By Theorem \ref{thm-centralizer}, for each periodic component $X$, a generating set $A$ of $C_\text{periodic}(f)$ can be computed explicitly. Given the topological complexity of the covering maps $\tl f^n:W\to X$ for all thick preimages of periodic components, it is straightforward to obtain an upper bound on the word length (in terms of the elements of $A$) of the generating set of $C_\text{thick}(f)$. By Theorem~\ref{th:hurwitz}, we can verify algorithmically, which of the words, whose length is under this bound, correspond to elements of $C_\text{thick}(f)$.
\end{proof}

Consider two equivalences on thick parts $\phi$ and $\psi$ between two Thurston maps $f$ and $g$. Then $\phi^{-1} \circ \psi$ 
is a self-equivalence of $f$. This yields the following.

\begin{lemma}
\label{lem:selfeq2}
Let $\phi$ be an equivalence on thick parts between two Thurston maps $f$ and $g$. Then any other equivalence can be written $\phi \circ l$ where 
$l \in C_\text{thick}(f)$.
\end{lemma}

\section{Algorithmic geometrization of thick parts}
The second and third authors proved the following \cite[Theorem 6.1]{SY}:
\begin{theorem}[\bf{Canonical geometrization}]\label{th:canonicalgeometrization}
  There exists an algorithm which for any Thurston map $f$ finds its canonical obstruction $\Gamma_f$.

Furthermore, let $\cF$ denote the collection of the first return maps of the canonical decomposition
of $f$ along $\Gamma_f$. Then  the algorithm outputs the following information:
\begin{itemize}
\item for every first return map with a hyperbolic orbifold, the unique (up to M{\"o}bius conjugacy)
marked rational map equivalent to it;
\item for every first return map of type $(2,2,2,2)$ the unique (up to affine conjugacy) affine map of the form $z \mapsto Az+b$ where $A \in \text{SL}_2(\ZZ)$ and $b\in \frac{1}{2}\Z^2$ with marked points which is equivalent to $f$ after quotient by the orbifold group $G$;
\item for every first return map which has a parabolic orbifold not of type $(2,2,2,2)$ the unique (up to  M{\"o}bius conjugacy) 
marked rational map map equivalent to it, which is a quotient of a complex affine map by the orbifold group.
\end{itemize}
\end{theorem}

\section{Extending equivalence from thick to thin parts}
The following is standard (see e.g. \cite{primer}):
\begin{proposition}
\label{prop:free abelian}
For every Thurston obstruction $\Gamma=\{\alpha_1,\ldots,\alpha_n\}$, the Dehn twists
$T_{\alpha_j},\; j=1\ldots n$ generate a free Abelian subgroup of $\text{PMCG}(S\setminus Q_f)$.
\end{proposition}
We write $\ZZ^\Gamma\simeq \ZZ^n$ to denote the subgroup generated by $T_{\alpha_j}$.

\noindent
 We will need the following straightforward generalization of  \cite[Proposition 7.7]{SY}:
 
 \begin{proposition} Let $f, g$ be equivalent Thurston maps. Let the pair $(\phi_1,\phi_2)$ realize the equivalence of the thick components of $f$ and $g$. Extend $\phi_1$ to a homeomorphism of the whole sphere $S^2\setminus Q_f$, defining it on the thin parts in an arbitrary fashion. 
   Then there exist  $m \in \ZZ^\Gamma$, $\psi \in C_{\text{thick}}$, and an equivalence pair $(h_1,h_2)$ for $f$, $g$ such that $h_1=\phi_1 \circ \psi \circ m$.
\label{p:NormalEquivalence}
\end{proposition}

Notice that if $h_1 \circ f = g\circ h_2$ where $h_1= \phi_1 \circ m_1$ for some $m_1 \in \ZZ^\Gamma$, then $h_2$ is homotopic to $\phi_1 \circ m_2$ for some other $m_2 \in \ZZ^\Gamma$. If $m_1=m_2$ then $h_1$ is homotopic to $h_2$ and these two homeomorphisms realize an equivalence between $f$ and $g$. Since we cannot check whether this happens for all elements of $\ZZ^\Gamma$  we will require the following proposition \cite[Proposition 7.8]{SY}:

\begin{proposition}
\label{p:finiteN} There exists explicitly computable $N \in \N$ such that if $n\in \ZZ^\Gamma$ where all coordinates of $n$ are divisible by $N$, then 

$$(\phi_1\circ (m_1+n)) \circ f = g\circ (\phi_2\circ( m_2+M_\Gamma n)),$$  
 whenever $$(\phi_1\circ m_1) \circ f = g\circ (\phi_2\circ m_2).$$  

\end{proposition}

\section{Checking Thurston equivalence.}
We are now ready to present an algorithm which checks whether two Thurston maps $f$ and $g$ are equivalent or not.

\textbf{Algorithm.}

\begin{enumerate}
\item Find the canonical obstructions $\Gamma_f=\{\alpha_1,\ldots,\alpha_n\}$ and $\Gamma_g=\{\beta_1,\ldots,\beta_n\}$ 
 (Theorem~\ref{th:canonicalgeometrization}). 
\item Check whether the cardinality of the canonical obstructions $\Gamma_f=\{\alpha_1,\ldots,\alpha_n\}$ and $\Gamma_g=\{\beta_1,\ldots,\beta_n\}$  is the same, and whether the corresponding Thurston matrices coincide. If not,
output {\bf maps are not equivalent} and halt. 
\item Denote the thin parts (decomposition annuli) of $f$ and $g$ by $A_i$ and $B_i$ respectively.
Construct the first return maps $\cF$ and $\cG$ of the periodic patched thick parts for $f$ and $g$ and geometrize them
(\thmref{th:canonicalgeometrization}). Are the geometrizations of $\cF$ and $\cG$ the same up to reordering of the components of the first return map? If not,
output {\bf maps are not equivalent} and halt. 
\item {\bf for all } permutations $\sigma\in S_n$ {\bf do}

  \setlength{\leftskip}{1.5cm}
\item   Is there a homeomorphism $$h_\sigma:S^2\setminus Q_f\to S^2\setminus Q_g$$ sending $A_i\to B_{\sigma(i)}$? If not, {\bf continue.}
\item   Is it true that for every periodic patched thick component $X$ of $f$ the geometrization of $\cF|_X$ is the same as the geometrization of
  $\cG_{h_\sigma(X)}$? 
  If not, {\bf continue.} 
\item   For all thick components $C_j^f$ check whether the Hurwitz classes of the patched coverings $$\tilde f:\wtl{C_j^f}\to \wtl{f(C_j^f)}\text{ and }\tilde g:\wtl{h_\sigma(C_j^f)}\to \wtl{g(h_\sigma(C_j^f))}$$ are the same (\thmref{th:hurwitz}).
 If not, {\bf continue.}

\item  Construct equivalence pairs $(\eta^X_0,\eta^X_1)$ between first return maps $\cF_X$ and $\cG_{h_\sigma(X)}$ of periodic patched thick components corresponding by $h_\sigma$ and the group  $C_\text{periodic}(g)$ of self-equivalences of $\cG$. If the maps of some pair are not equivalent, {\bf continue.}

\item  Find an equivalence between first return maps $\cF_X$ and $\cG_{h_\sigma(X)}$ in the form $\phi_X=\psi\circ \eta^X_0$ with $\psi\in C_{\text{periodic}}(g)$ that can be lifted 
via branched covers $\tilde f$ and $\tilde g$ to every preimage of every thick component and preserves the set of marked points. Since $C_{\text{thick}}$ is a finite index subgroup of $C_{\text{periodic}}(g)$, this is a finite check (for representatives of each coset), which can be carried out algorithmically by (Theorem~\ref{th:hurwitz} and Lemma~\ref{lem:selfeq2}). If not possible, {\bf continue.}


\item  Lift the equivalences, to obtain a homeomorphism  $\phi_1$ defined on all thick parts.

\item  Compute $C_{\text{thick}}(g)$ (Lemma~\ref{lem:selfeq2}).

\item  Pick some initial homemorphisms $a_i:A_i \to B_{\sigma(i)}$ so that the boundary  values agree with $\phi_1$. This defines $\phi_1$ on the whole sphere.

\item    Find the set of vectors $m_1 \in \ZZ^\Gamma$ with coordinates between 0 and $N-1$, where $N$ is as in Proposition~\ref{p:finiteN} such that
   $h_1=\phi_1 \circ m_1$ lifts through $f$ and $g$ so that $$( \phi_1 \circ m_1) \circ f =g \circ h_2.$$ { \bf For all} vectors $m_1$ in this set {\bf do}
 \setlength{\leftskip}{2.5cm}
 
 \item By the discussion above $h_2=m_2 \circ \phi_2 $ with $m \in \ZZ^\Gamma$. Compute $m_2$.
\item   Find the finite index subgroup  $G_1$ of $ C_\text{thick}(f) $ of all elements $\psi$ such that $\psi \circ h_1$ lifts through $f$ and $g$ (Lemma~\ref{lem:selfeq1}). 
 
\item  For every $\psi\in G_1$ we have $$\psi  \circ h_1  \circ f = g \circ   n_\psi \circ \psi \circ  h_2 $$ where $n_\psi \in \ZZ^\Gamma$.  The map $\psi \mapsto n_\psi-m_2$ is a homomorphism.

\item   Similarly, find the finite index subgroup $G_2$ of $\ZZ^\Gamma$ of all elements $k$ such that $k\circ h_1$ lifts through $f$ and $g$. For every $k\in G_1$ we have $$k  \circ h_1  \circ f = g \circ   n_k  \circ  h_2 $$ where $n_k \in \ZZ^\Gamma$.  The map $k \mapsto n_k-m_2$ is also a homomorphism (linear).  

\item  Using generators of $ C_\text{thick}(g) $ construct $\psi$ and $k$ such that $ k+m_1=n_k+n_\psi+m_2$. If $m_1-m_2$ is not in the image of $n_k+n_\psi-k$, \textbf{ continue.}

\item  Output {\bf maps are equivalent} and $\psi \circ k \circ h_1$; halt.
  
  \setlength{\leftskip}{1.5cm}
  
\item {\bf end do}  

\setlength{\leftskip}{0cm}
\item {\bf end do}
\item output {\bf maps are not equivalent} and halt. 
\end{enumerate}
If the algorithm exits on step 17, then $\phi\circ h_0$ realizes the equivalence between $f$ and $g$, by construction. Otherwise, no such equivalence exists, by Proposition~\ref{p:NormalEquivalence}, and thus the above algorithm satisfies the conditions of our main theorem.

\bibliographystyle{amsalpha}

\bibliography{biblio}

\end{document}